\documentclass[a4paper,notitlepage, reqno]{amsart}
\usepackage{fullpage}
\usepackage[colorlinks=true,linkcolor=blue,unicode, psdextra]{hyperref}
\usepackage{color}
\usepackage{amsthm}
\usepackage{amsmath}
\usepackage{amssymb}
\usepackage{enumitem}
\usepackage{mathtools}
\usepackage{mathrsfs}
\usepackage[utf8]{inputenc}

\newtheorem*{theorem*}{Theorem}
\newtheorem*{proposition*}{Proposition}
\newtheorem*{lemma*}{Lemma}
\newtheorem*{corollary*}{Corollary}
\newtheorem*{remark*}{Remark}

\newtheorem{theorem}{Theorem}[section]
\newtheorem{lemma}[theorem]{Lemma}
\newtheorem{corollary}[theorem]{Corollary}
\newtheorem{proposition}[theorem]{Proposition}
\theoremstyle{definition}

\newtheorem{remark}[theorem]{Remark}

\newcommand{\Spec}{\textrm{Spec}}
\newcommand\vect[2]{#1_1,\,\ldots,\, #1_{#2}}
\def\vect#1#2{{#1}_1, \, \ldots, \, {#1}_{#2}}

\allowdisplaybreaks

\numberwithin{equation}{section}
\setenumerate[1]{label={\alph*)}}

\date{}
\title{On the integrality of étale extensions of polynomial rings}
\author{Lázaro Orlando Rodríguez Díaz}
\address{Instituto de Matem\'{a}tica, 
	Universidade Federal do Rio de Janeiro, RJ, Brazil.}
\email{lazarord@im.ufrj.br}	

\begin{document}
	
\begin{abstract}
Motivated by a valuation theorem, recently obtained by Rangachev, we study the \'etale extensions $A\subset B$ of polynomial rings over an algebraically closed field of characteristic zero, such that the integral closure $\overline{A}$ is a primary $\overline{A}$-submodule of $B$. We prove that in this case $\overline{A}$ has infinite cyclic divisor class group, where the generator is a prime divisor equal to the complement of $\textrm{Spec}(B)$ in $\textrm{Spec}(\overline{A})$. Moreover, this prime divisor coincides with the ramification divisor of the finite extension $A\subset \overline{A}$. In this situation we carry out Wright's geometric approach for two-dimensional non-integral \'etale extensions. It follows from the work of Miyanishi that $\textrm{Spec}(\overline{A})$ is a smooth affine surface. We show that $\textrm{Spec}(\overline{A})$ is an $\mathbb{A}^{1}$-bundle over $\mathbb{P}^{1}$, more precisely a Danilov-Gizatullin surface of index three. Based on Wright's analysis of which of these affine surfaces can factorize an \'etale morphism of the complex affine plane and his description of its affine coordinate rings, we prove that under the strong assumption that $\overline{A}$ is always a primary $\overline{A}$-submodule of $B$, any two-dimensional complex \'etale extension is integral.
\end{abstract}
\maketitle

\section{Introduction}
The problem we want to address is the following: how far is the integral closure $\overline{A}$ of $A$ from $B$, assuming that an $A=k[\vect y n]\subset B=k[\vect x n]$ is an \'{e}tale extension of polynomial rings  over an algebraically closed field $k$ of characteristic zero? It is well-known that, if $B$ is integral over $A$ then $B=A$, \cite[Theorem 47]{Wang80}, \cite[Theorem 3.3]{Wright81}. Then the problem considered makes sense.  On the other hand, it is known that the integral closure of $A$ in $\mathrm{Frac(B)}$ is the intersection $\overline{A}=\cap \mathcal{V}$ of the family of valuations rings $\mathcal{V}$ of $\mathrm{Frac(B)}$ that contains $A$, \cite[Theorem 19.8]{Gilmer_92}. Therefore, we can write $\overline{A}=B\cap \left(\cap \mathcal{V}\right)$, where the intersection is taken over all valuations rings $\mathcal{V}$ of $\mathrm{Frac}(B)$ that contain $A$ and does not contain $B$. The question can be reformulated as: how many valuations rings are there?, can we describe them? 

Recently, Rangachev \cite{Rangachev} proved that if $A \subset B$ are integral domains, $A$ is Noetherian and $B$ is a finitely generated $A$-algebra then there are a finite number of uniquely determined discrete valuation rings $\mathcal{V}_{i}$ each of which is a localization of $\overline{A}$ at a height one prime ideal, that is,  $\overline{A}=B\cap \left(\cap_{i=1}^{r} \mathcal{V}_{i}\right)$, $\mathcal{V}_{i}=\overline{A}_{\mathfrak{p}_{i}}$, $\mathrm{ht}(\mathfrak{p}_{i})=1$. In the present work we study the simplest possible case, that is, we assume that there is only one discrete valuation ring in the above decomposition, $\overline{A}=B\cap \overline{A}_{\mathfrak{p}}$. Due to Rangachev's results this is equivalent to assume that $\overline{A}$ is a primary $\overline{A}$-submodule of $B$, see Lemma \ref{main-lemma}. 

We prove in Proposition \ref{prop_main_proposition} that if $A=k[\vect y n]\subset B=k[\vect x n]$ is an \'{e}tale extension of polynomial rings over an algebraically closed field $k$ of characteristic zero, such that $\overline{A}\neq B$ and $\overline{A}$ is a primary $\overline{A}$-submodule of $B$, then $\operatorname{Spec}(\overline{A})\setminus\operatorname{Spec}(B)=V\left(\mathfrak{p}\right)$, where $\mathfrak{p}$ is a prime ideal of height one in $\overline{A}$, and the divisor class group $\mathrm{Cl}(\operatorname{Spec}(\overline{A}))\cong \mathbb{Z}\cong \left\langle V\left(\mathfrak{p}\right)\right\rangle$. Moreover, the prime divisor $V\left(\mathfrak{p}\right)$ coincides with the ramification divisor of the finite extension $A\subset \overline{A}$, see Corollary \ref{cor_ramification}.

In Section \ref{sec_dimension_two} we apply this proposition to carry out Wright's geometric approach \cite{Wright97} for two-dimensional non-integral étale extensions. Wright proved that, given an \'{e}tale extension $A=\mathbb{C}[p,q]\subset B=\mathbb{C}[x,y]$ (that satisfies some assumptions, see Remark \ref{rmk_Wright_assumption}) such that $\overline{A}\neq B$, then we can construct a normal affine variety $V$, containing $\Spec(B)$ as an open subvariety, such that $V$ admits a map to $\mathbb{P}^{1}$ making it an $\mathbb{A}^{1}$-bundle except over one point of $\mathbb{P}^{1}$ whose fiber is $V\setminus\Spec(B)$ set-theoretically. We proved that if $\overline{A}$ is a primary $\overline{A}$-submodule of $B$ the obstacles encountered by Wright to prove that $V$ is an $\mathbb{A}^{1}$-bundle over $\mathbb{P}^{1}$ can be surpassed. The difficulties are: $\operatorname{Spec}(\overline{A})$ could have singularities, and even if it did not have any singularities, the fiber at infinity  may be non-reduced. Due to Miyanishi's work, Proposition \ref{prop_main_proposition} implies that  $\operatorname{Spec}(\overline{A})$ is a smooth surface, see Proposition \ref{prop-smooth_integral_closure}. We proved the following result:
\begin{theorem*}
	Let $A=\mathbb{C}[p,q]\subset B=\mathbb{C}[x,y]$ be an \'{e}tale extension such that $\overline{A}\neq B$ and $\overline{A}$ is a primary $\overline{A}$-submodule of $B$, then $\operatorname{Spec}(\overline{A})$ is a smooth affine surface that admits the structure of an $\mathbb{A}^{1}$-bundle over $\mathbb{P}^{1}$, $\pi: \operatorname{Spec}(\overline{A})\to \mathbb{P}^{1}$ such that the open embedding of $\operatorname{Spec}(B)\cong \mathbb{A}^{2}$ in $\operatorname{Spec}(\overline{A})$ coincides with the complement of the fiber of $\pi$ over the infinity and $\pi\vert_{\operatorname{Spec}(B)}:\mathbb{A}^{2}\to \mathbb{P}^{1}\setminus\{\infty\}=\mathbb{A}^{1}$ is one of the standard projections. Moreover, $\operatorname{Spec}(\overline{A})$ is a Danilov-Gizatullin surface of index three.
\end{theorem*}
That $\operatorname{Spec}(\overline{A})$ is a Danilov-Gizatullin surface of index three, see Proposition \ref{prop_V_is_DG}, follows because we deduced from Proposition \ref{prop_main_proposition} that the Picard group of $\operatorname{Spec}(\overline{A})$ is generated by its canonical class, see Corollary \ref{cor_canonical_divisor}. Among the Danilov-Gizatullin surfaces the only one whose Picard group is generated by the canonical class are those of index three.

We determine the explicit form of the integral closure using Wright's explicit description of the coordinates ring of an $\mathbb{A}^{1}$-bundle over $\mathbb{P}^{1}$ and Wright's characterization of when such bundles can factorize an étale morphism of the affine plane. More precisely, we prove that if $A=\mathbb{C}[p,q]\subset B=\mathbb{C}[x,y]$ is an \'{e}tale extension such that $\overline{A}\neq B$ and $\overline{A}$ is a primary $\overline{A}$-submodule of $B$, then $\overline{A}=\mathbb{C}[y, xy, x^{2}y, x^{3}y+\alpha x]$, where $\alpha\in \mathbb{C}$, $\alpha\neq 0$; in particular $p,q\in \mathbb{C}[y, xy, x^{2}y, x^{3}y+\alpha x]$, see Proposition \ref{prop_description_integral_closure}. Wright noted that $\mathbb{C}[y, xy, x^{2}y, x^{3}y+\alpha x]$ is a graded algebra with $\textrm{deg}\;x=-1$ and $\textrm{deg}\;y=2$. Using this observation we show that $\mathbb{C}[y, xy, x^{2}y, x^{3}y+\alpha x]$ does not contain polynomials that are regular in both variables, see Lemma \ref{lmm-non_regular_element}. Nevertheless, we can always choose a linear automorphism that makes, for example, the polynomial $p$ regular in both variables, from there we arrive at the following result:
\begin{theorem*}
	Suppose that $\overline{A}$ is a primary $\overline{A}$-submodule of $B$ for every étale extension $A=\mathbb{C}[p,q]\subset B=\mathbb{C}[x,y]$. Then every \'{e}tale extension of polynomial rings  $\mathbb{C}[p,q]\subset \mathbb{C}[x,y]$ is integral.
\end{theorem*}

\section{Étale extensions whose integral closure is a primary submodule}	
Recently, Rangachev proved a valuation theorem for Noetherian rings under quite general hypothesis. 
\begin{theorem}{\cite[Theorem 1.1]{Rangachev}}\label{Rangachev_theorem}
	Let $A\subset B$ be integral domains, suppose $A$ is Noetherian and $B$ is a finitely generated $A$-algebra. Denote by $\overline{A}$ the integral closure of $A$ in $B$. Then one of the following holds:
	\begin{itemize}
		\item[(i)] $\overline{A}=B$;
		\item[(ii)] $\mathrm{Ass}_{\overline{A}}\left(B/\overline{A}\right)=\{(0)\}$;\item[(iii)] $\mathrm{Ass}_{\overline{A}}\left(B/\overline{A}\right)=\{\mathfrak{p}_{1}, \ldots, \mathfrak{p}_{r}\}$, $\mathrm{ht}(\mathfrak{p}_{i})=1$ for all $i$, and $\overline{A}=B\cap \left(\cap_{i=1}^{r} \mathcal{V}_{i}\right)$, where $\mathcal{V}_{i}=\overline{A}_{\mathfrak{p}_{i}}$.
	\end{itemize} 
\end{theorem} 
Also, under the same hypothesis, he provides a description of  the support of $B/\overline{A}$ as a closed set.
\begin{proposition}{\cite[Proposition 2.5]{Rangachev}}\label{Rangachev_proposition}
	Assume we are in the situation of Theorem \ref{Rangachev_theorem}. Denote by $\mathcal{I}_{B/\overline{A}}$ the intersection of all elements in $\mathrm{Ass}_{\overline{A}}\left(B/\overline{A}\right)$. Then $\mathrm{Supp}_{\overline{A}}\left(B/\overline{A}\right)=V\left(\mathcal{I}_{B/\overline{A}}\right)$.
\end{proposition}	
We want to study the case of Theorem \ref{Rangachev_theorem} where there is only one discrete valuation ring of $\overline{A}$ that does not contain $B$, that is, $\overline{A}=B\cap \overline{A}_{\mathfrak{p}}$. In that situation we have the following:	
\begin{lemma}\label{main-lemma}
Let $A\subset B$ be integral domains such that $A$ is Noetherian and $B$ is a finitely generated $A$-algebra. Denote by $\overline{A}$ the integral closure of $A$ in $B$. Suppose that $\overline{A}\ne B$ and $\mathrm{Frac}(\overline{A})=\mathrm{Frac}(B)$. Then $\overline{A}$ is a primary $\overline{A}$-submodule of $B$ if and only if $\mathrm{Ass}_{\overline{A}}\left(B/\overline{A}\right)=\{\mathfrak{p}\}$, $\mathrm{ht}(\mathfrak{p})=1$, $\overline{A}=B\cap \overline{A}_{\mathfrak{p}}$ and $\mathrm{Supp}_{\overline{A}}\left(B/\overline{A}\right)=V\left(\mathfrak{p}\right)$. 
\end{lemma}
\begin{proof}
As $\overline{A}\ne B$, this rule out the case (i) above. We always have that $\mathrm{Ass}_{\overline{A}}\left(B/\overline{A}\right)\subset \mathrm{Supp}_{\overline{A}}\left(B/\overline{A}\right)$, if $\mathrm{Ass}_{\overline{A}}\left(B/\overline{A}\right)=\{(0)\}$, then $(0) \in \mathrm{Supp}_{\overline{A}}\left(B/\overline{A}\right)$;  however, this contradicts that $\mathrm{Frac}(\overline{A})=\mathrm{Frac(B)}$. Then by Rangachev's Theorem \ref{Rangachev_theorem} we are left with the third possibility, that is,  $\mathrm{Ass}_{\overline{A}}\left(B/\overline{A}\right)=\{\mathfrak{p}_{1}, \ldots, \mathfrak{p}_{r}\}$, $\mathrm{ht}(\mathfrak{p}_{i})=1$, $\forall i$. The equivalence now follows from the definition of primary submodule \cite[Definition 8.A]{Matsumura80} and Rangachev's results.
\end{proof}	
We are going to use this lemma to prove the following proposition:
\begin{proposition}\label{prop_main_proposition}
Let $k$ be an algebraically closed field of characteristic zero. Let  $A=k[\vect y n]\subset B=k[\vect x n]$ be an \'{e}tale extension of polynomial rings (each in $n$ indeterminates). Denote by $\overline{A}$ the integral closure of $A$ in $B$. Suppose that $\overline{A}\neq B$ and that $\overline{A}$ is a primary $\overline{A}$-submodule of $B$. Then $\operatorname{Spec}(\overline{A})\setminus\operatorname{Spec}(B)=V\left(\mathfrak{p}\right)$, where $\mathfrak{p}$ is a prime ideal of height one in $\overline{A}$, and the divisor class group $\mathrm{Cl}(\operatorname{Spec}(\overline{A}))\cong \mathbb{Z}\cong \left\langle V\left(\mathfrak{p}\right)\right\rangle$.
\end{proposition}

\begin{proof}
The extension $A\subset B$ is \'{e}tale, so it is quasi-finite and by Zariski's Main Theorem \cite[Corollaire 2, p. 42]{Raynaud}, the morphism $\operatorname{Spec}(B)\to \operatorname{Spec}(\overline{A})$ is an open immersion. Also, the \'{e}taleness of the extension implies that $\mathrm{Frac(A)}\subset\mathrm{Frac(B)}$  is a finite separable algebraic extension of fields, then $\overline{A}$ is a finite module over $A$, in particular $\overline{A}$ is Noetherian. It is clear that $\mathrm{Frac}(\overline{A})=\mathrm{Frac}(B)$. 

Then the \'{e}tale extension $A\subset B$ satisfies the hypothesis of Lemma \ref{main-lemma}. We are assuming that $\overline{A}$ is a primary $\overline{A}$-submodule of $B$, therefore, there exist a prime ideal $\mathfrak{p}$ of height one in $\overline{A}$ such that $\overline{A}=B\cap \overline{A}_{\mathfrak{p}}$, $\mathrm{Ass}_{\overline{A}}\left(B/\overline{A}\right)=\{\mathfrak{p}\}$, and $\mathrm{Supp}_{\overline{A}}\left(B/\overline{A}\right)=V\left(\mathfrak{p}\right)$. From this we can derive the following consequences:
\begin{itemize}[leftmargin=*]	
\item[-]$\operatorname{Spec}(\overline{A})\setminus\operatorname{Spec}(B)=V\left(\mathfrak{p}\right)$, where we identify $\operatorname{Spec}(B)$ with its image in $\operatorname{Spec}(\overline{A})$. In fact, denote by $g:\operatorname{Spec}(B)\hookrightarrow \operatorname{Spec}(\overline{A})$ the open immersion. We are going to show that $\operatorname{Im}g=\operatorname{Spec}(\overline{A})\setminus \mathrm{Supp}_{\overline{A}}\left(B/\overline{A}\right)$. Suppose that $\mathfrak{p}\in \operatorname{Im}g\subset \operatorname{Spec}(\overline{A})$, that is, there exist $\mathfrak{q}\in \operatorname{Spec}(B)$ such that $\mathfrak{p}=\mathfrak{q}\cap \overline{A}$. As $g$ is an open immersion we have that $\overline{A}_{\mathfrak{p}}=B_{\mathfrak{q}}=B_{\mathfrak{p}}$, see for example \cite[Exercise 2.3.9]{Ford17}, therefore $\mathfrak{p}\notin \mathrm{Supp}_{\overline{A}}\left(B/\overline{A}\right)$. Conversely, suppose that $\mathfrak{p}\in \operatorname{Spec}(\overline{A})$ and $\mathfrak{p}\notin \mathrm{Supp}_{\overline{A}}\left(B/\overline{A}\right)$, that is, $\overline{A}_{\mathfrak{p}}=B_{\mathfrak{p}}$. Then $B \subset B_{\mathfrak{p}}=\overline{A}_{\mathfrak{p}}$, define the prime ideal $\mathfrak{r}:=\mathfrak{p}\overline{A}_{\mathfrak{p}}\cap B$, then $\mathfrak{r}\in \operatorname{Spec}(B)$ and $\mathfrak{r}\cap \overline{A}=\mathfrak{p}$, therefore $\mathfrak{p}\in \operatorname{Im}g$. The assertion now follows because we already know that $\mathrm{Supp}_{\overline{A}}\left(B/\overline{A}\right)=V\left(\mathfrak{p}\right)$.
\item[-]$\mathrm{Cl}(\operatorname{Spec}(\overline{A}))\cong \mathbb{Z}\cong \left\langle V\left(\mathfrak{p}\right)\right\rangle$. In fact, $\operatorname{Spec}(\overline{A})$ is a noetherian normal affine scheme, $V(\mathfrak{p})$ is an irreducible closed subset of codimension one, and $\operatorname{Spec}(\overline{A})\setminus\operatorname{Spec}(B)=V\left(\mathfrak{p}\right)$, then by Hartshorne \cite[Chap. II, Prop.6.5]{Hartshorne77} there is an exact sequence $\mathbb{Z}\to \mathrm{Cl}(\operatorname{Spec}(\overline{A}))\to \mathrm{Cl}(\operatorname{Spec}(B))\to 0$, where the first map is defined by $1\to 1\cdot V(\mathfrak{p})$.  Note that $\mathrm{Cl}(\operatorname{Spec}(B))=0$ because $B$ is a unique factorization domain, then $\mathrm{Cl}(\operatorname{Spec}(\overline{A}))$ is generated by $V\left(\mathfrak{p}\right)$. On the other hand, $V\left(\mathfrak{p}\right)$ cannot be of finite order in $\mathrm{Cl}(\operatorname{Spec}(\overline{A}))$, because it is known \cite[Proposition~6]{Angermuller83} that if $\mathrm{Cl}(\operatorname{Spec}(\overline{A}))$ were a torsion group (i.e., $\overline{A}$ were almost factorial) then $A=B$, but we are assuming that $\overline{A}\neq B$. We have proved the assertion.
\end{itemize}
\end{proof}	
Suppose we are in the situation of Proposition \ref{prop_main_proposition}. Then the étale morphism $\mathbb{A}^{n}=\operatorname{Spec}(B)\xrightarrow{f} \mathbb{A}^{n}=\operatorname{Spec}(A)$ induced by the étale extension $A\subset B$ factorizes as $f=hg$, $\mathbb{A}^{n}=\operatorname{Spec}(B)\xhookrightarrow{g} \operatorname{Spec}(\overline{A}) \xrightarrow{h} \mathbb{A}^{n}=\operatorname{Spec}(A)$, where $g$ is an open immersion and $h$ is finite. We denote by $R_{h}$ the ramification divisor of the morphism $h$, since $f$ is étale, we have that $R_{h}\subset \operatorname{Spec}(\overline{A})\setminus\operatorname{Spec}(B)=V\left(\mathfrak{p}\right)$. By the purity of the branch locus it follows that $R_{h}=V(\mathfrak{p})$. Then we have proved:
\begin{corollary}\label{cor_ramification}
Assume we are in the situation of Proposition \ref{prop_main_proposition}. Then the ramification divisor $R_{h}\subset \operatorname{Spec}(\overline{A})$ of the finite morphism $\operatorname{Spec}(\overline{A}) \xrightarrow{h} \mathbb{A}^{n}=\operatorname{Spec}(A)$ coincides with the prime divisor $V\left(\mathfrak{p}\right)=\operatorname{Spec}(\overline{A})\setminus\operatorname{Spec}(B)$.
\end{corollary}

\begin{corollary}\label{cor_canonical_divisor}
Assume we are in the situation of Proposition \ref{prop_main_proposition}. Suppose $\operatorname{Spec}(\overline{A})$ is smooth and denote by $K_{\operatorname{Spec}(\overline{A})}$ its canonical divisor. Then $K_{\operatorname{Spec}(\overline{A})}=V\left(\mathfrak{p}\right)$, that is, the Picard group of $\operatorname{Spec}(\overline{A})$ is generated by the canonical class.
\end{corollary}
\begin{proof}
We have that $K_{\operatorname{Spec}(\overline{A})}=h^{*}(K_{\mathbb{A}^{2}})+ R_{h}$ by the ramification formula \cite[Theorem 5.5]{Iitaka82} applied to the dominant morphism $\operatorname{Spec}(\overline{A}) \xrightarrow{h} \mathbb{A}^{2}=\operatorname{Spec}(A)$. Then $K_{\operatorname{Spec}(\overline{A})}=R_{h}$. The conclusion now follows from Corollary \ref{cor_ramification}.	
\end{proof}

\section{Two-dimensional étale extensions}\label{sec_dimension_two}
	
In this section $k=\mathbb{C}$, the field of complex numbers. We are going to apply Proposition \ref{prop_main_proposition} to carry out Wright's geometric approach for two-dimensional non-integral étale extensions \cite{Wright97} in the case $\overline{A}$ is a primary $\overline{A}$-submodule of $B$. Wright proved that if we start with an étale extension $A=\mathbb{C}[p,q]\subset B=\mathbb{C}[x,y]$ that is not integral, it is possible to construct a normal affine variety that admits the structure of an $\mathbb{A}^{1}$-fibration over $\mathbb{P}^{1}$, more precisely: 

\begin{theorem}\cite[Theorem 4.3]{Wright97}\label{thm_Wright_bundle}
	Let $A=\mathbb{C}[p,q]\subset B=\mathbb{C}[x,y]$ be an \'{e}tale extension such that $\overline{A}\neq B$, then there exist a normal affine variety $V$ containing $U=\mathbb{A}^{2}$ as an open subvariety having the following properties:
	\begin{itemize}
		\item[i)] $F=V\setminus U$ is a rational curve whose normalization is $\mathbb{A}^{1}$ and each singular point of $F$ has one-point desingularization.
		\item[ii)]there is a map $\pi: V\to \mathbb{P}^{1}$ such that $F$ is the set-theoretic fiber of a point $z\in\mathbb{P}^{1}$, and the restriction map $\pi\vert_{U}:U\to \mathbb{P}^{1}\setminus\{z\}=\mathbb{A}^{1}$ is the projection onto a coordinate line.
		\item[iii)] there is a map $f: V\to \mathbb{A}^{2}$ such that $f\vert_{U}$ is étale.  
	\end{itemize}
	Moreover, if we assume that $\overline{A}$ is smooth, then $V$ can be chosen to be smooth and $F\cong\mathbb{A}^{1}$. 	Even more, \cite[Remark p. 605]{Wright97} if $F$ has multiplicity one in the fiber, that is, $\pi^{-1}(z)=\mathbb{A}^{1}$ scheme-theoretically, then $V$ is an $\mathbb{A}^{1}$-bundle over $\mathbb{P}^{1}$ via the map $\pi: V\to \mathbb{P}^{1}$.
\end{theorem}
	
In order to get an $\mathbb{A}^{1}$-bundle over $\mathbb{P}^{1}$ we first need to prove that the integral closure is smooth. We are going to show that this is the case if $\overline{A}$ is a primary $\overline{A}$-submodule of $B$. To achieve this we apply the following result of Gurjar and Miyanishi. In fact, the specific case we need from this result, that is, when $X=\mathbb{A}^{2}$, was proved by Miyanishi in \cite[Chap. I, Sect. 6]{Miyanishi_81}.

\begin{theorem}{\cite[Theorem~1]{Gurjar-Miyanishi-08}}\label{thm-Gurjar_Miyanishi}
Let $\varphi: X_{u}\to X_{l}$ be an \'{e}tale morphism, where $X_{u}=X$ and $X_{l}=X$ are two copies of an irreducible normal affine surface $X$ defined over $\mathbb{C}$. Let $\widetilde{X}$ be the normalization of $X_{l}$ in the function field of $X_{u}$. Then $\widetilde{X}\setminus X_{u}$ is a disjoint union of irreducible curves $C_{i}$, each of which is isomorphic to $\mathbb{A}^{1}$. Further, each singularity of $\widetilde{X}$ which is not contained in $X_{u}$ is a cyclic quotient singularity. Any irreducible component of $\widetilde{X}\setminus X_{u}$ contains at most one such singular point.
\end{theorem} 	
	
Consider now an \'{e}tale extension $A=\mathbb{C}[p,q]\subset B=\mathbb{C}[x,y]$, or equivalently, the corresponding \'{e}tale morphism $f=(p,q):\mathbb{A}^{2}=\operatorname{Spec}(B)\to \mathbb{A}^{2}=\operatorname{Spec}(A)$. Suppose that $\overline{A}\neq B$ and that $\overline{A}$ is a primary $\overline{A}$-submodule of $B$, by Proposition \ref{prop_main_proposition} we have that $\operatorname{Spec}(\overline{A})\setminus\operatorname{Spec}(B)=V\left(\mathfrak{p}\right)$ is reduced and irreducible. Applying the Gurjar-Miyanishi's Theorem \ref{thm-Gurjar_Miyanishi} to the morphism $f=(p,q):\mathbb{A}^{2}=\operatorname{Spec}(B)\to \mathbb{A}^{2}=\operatorname{Spec}(A)$ implies that $\operatorname{Spec}(\overline{A})\setminus\operatorname{Spec}(B)\cong\mathbb{A}^1$ scheme-theoretically, and consequently $\operatorname{Spec}(\overline{A})$ has no singularities at all. We have proved the following:
	
\begin{proposition}\label{prop-smooth_integral_closure}
Let  $A=\mathbb{C}[p,q]\subset B=\mathbb{C}[x,y]$ be an \'{e}tale extension of polynomial rings. Denote by $\overline{A}$ the integral closure of $A$ in $B$. Suppose that $\overline{A}$ is a primary $\overline{A}$-submodule of $B$. Then $\overline{A}$ is a smooth $\mathbb{C}$-algebra.
\end{proposition}

Under the hypothesis that $\overline{A}$ is a primary $\overline{A}$-submodule of $B$ it follows now from Wright's construction that the variety $V$ in Wright's Theorem \ref{thm_Wright_bundle} is exactly $\operatorname{Spec}(\overline{A})$, 
and it admits a map to $\mathbb{P}^{1}$ making it an $\mathbb{A}^{1}$-bundle except over one point of $\mathbb{P}^{1}$ whose fiber is exactly $F=\operatorname{Spec}(\overline{A})\setminus\operatorname{Spec}(B)\cong\mathbb{A}^{1}$, now scheme-theoretically.  Therefore, we have proved that $\operatorname{Spec}(\overline{A})$ has the structure of an $\mathbb{A}^{1}$-bundle over $\mathbb{P}^{1}$, such that the open embedding of $\operatorname{Spec}(B)\cong \mathbb{A}^{2}$ in $\operatorname{Spec}(\overline{A})$ coincides with the complement of the fiber of $\pi$ over the infinity and $\pi\vert_{\operatorname{Spec}(B)}:\mathbb{A}^{2}\to \mathbb{P}^{1}\setminus\{\infty\}=\mathbb{A}^{1}$ is one of the standard projections. Nevertheless, we must make the following observation:

\begin{remark}\label{rmk_Wright_assumption}
In his Theorem \ref{thm_Wright_bundle}, Wright assumed that the curves $p=0$ and $q=0$  each have two points at infinity in $\mathbb{P}^{2}$ that coincide with the points at infinity on the lines $x=0$ and $y=0$. He need this in his proof to accommodate the variety $V$ within the blowups performed in the resolution of the birational map associated to the étale extension. It is always possible to make such an assumption at the cost of modifying the original étale extension. Its means that if we start with an arbitrary non-integral \'{e}tale extension $A\subset B$, then we probably need to change to a different non-integral étale extension $A'\subset B'$ by means of an automorphism, if we want to construct $V$ from $\operatorname{Spec}(\overline{A'})$ as Wright does. 
\end{remark}

Let's prove using Proposition \ref{prop_main_proposition} that when $\overline{A}$ is a primary $\overline{A}$-submodule of $B$, we don't need to make Wright's assumptions. Proving it this way will be useful later to explicitly describe the integral closure and to prove Theorem \ref{thm-etale_implies_integral}.

\begin{theorem}\label{thm_p1_bundle}
Let $A=\mathbb{C}[p,q]\subset B=\mathbb{C}[x,y]$ be an \'{e}tale extension such that $\overline{A}\neq B$ and $\overline{A}$ is a primary $\overline{A}$-submodule of $B$, then $\operatorname{Spec}(\overline{A})$ is a smooth affine surface that admits the structure of an $\mathbb{A}^{1}$-bundle over $\mathbb{P}^{1}$, $\pi: \operatorname{Spec}(\overline{A})\to \mathbb{P}^{1}$ such that the open embedding of $\operatorname{Spec}(B)\cong \mathbb{A}^{2}$ in $\operatorname{Spec}(\overline{A})$ coincides with the complement of the fiber of $\pi$ over the infinity and $\pi\vert_{\operatorname{Spec}(B)}:\mathbb{A}^{2}\to \mathbb{P}^{1}\setminus\{\infty\}=\mathbb{A}^{1}$ is one of the standard projections.
\end{theorem}
\begin{proof}
	By Proposition \ref{prop_main_proposition} and Proposition \ref{prop-smooth_integral_closure} we have that $\operatorname{Spec}(\overline{A})$ is a smooth affine surface such that $\operatorname{Spec}(\overline{A})\setminus\operatorname{Spec}(B)\cong\mathbb{A}^1$ scheme-theoretically, lets denote this affine line by $F:=\operatorname{Spec}(\overline{A})\setminus\operatorname{Spec}(B)$. On the other hand, $\operatorname{Spec}(B)\cong \mathbb{A}^{2}$ admits two $\mathbb{A}^{1}$-fibrations over $\mathbb{A}^{1}$ given by the standard projections $\pi_{i}: \operatorname{Spec}(B)\to \mathbb{A}^{1}$, $i=1,2$, corresponding to the containments $\mathbb{C}[x]\subset \mathbb{C}[x,y]$ and $\mathbb{C}[y]\subset \mathbb{C}[x,y]$ respectively. It is known \cite[Chap. I, Lemma 4.3 ]{Miyanishi_81} that each $\pi_{i}$ extends to an $\mathbb{A}^{1}$-fibration $\widetilde{\pi}_{i}: \operatorname{Spec}(\overline{A})\to C$ where $C\cong\mathbb{A}^{1}$ or $C\cong\mathbb{P}^{1}$. 
	
	We are going to argue as in \cite[Lemma 5.2]{Gurjar_Masuda_Miyanishi_Russell_08} to show that at least one of the  $\pi_{i}$'s extends to an $\mathbb{A}^{1}$-fibration over $\mathbb{P}^{1}$. Suppose that both fibrations $\pi_{1}$ and $\pi_{2}$ extends to $\mathbb{A}^{1}$-fibrations over $\mathbb{A}^{1}$, that is, $\widetilde{\pi}_{1}: \operatorname{Spec}(\overline{A})\to \operatorname{Spec}(\mathbb{C}[x])$  and $\widetilde{\pi}_{2}: \operatorname{Spec}(\overline{A})\to \operatorname{Spec}(\mathbb{C}[y])$. Then $F$ is contained in some fiber of $\widetilde{\pi}_{1}$ and in some fiber of $\widetilde{\pi}_{2}$. This means that $x$ and $y$ are constant along $F$. Therefore, any function $f\in \Gamma(\operatorname{Spec}(\overline{A}),\mathcal{O}_{\operatorname{Spec}(\overline{A})})\subset \Gamma(\operatorname{Spec}(B),\mathcal{O}_{\operatorname{Spec}(B)})=\mathbb{C}[x,y]$ is also constant along $F$. This leads to a contradiction, because given any two points of $F$ there is a regular function of $\operatorname{Spec}(\overline{A})$ that separates them.
	
	Then at least one of the $\pi_{i}$'s extends to an $\mathbb{A}^{1}$-fibration over $\mathbb{P}^{1}$. Without loss of generality, suppose that it is $\pi_{1}$ that extends to an $\mathbb{A}^{1}$-fibration  $\widetilde{\pi}_{1}: \operatorname{Spec}(\overline{A})\to \mathbb{P}^{1}$. Then $\widetilde{\pi}_{1}^{-1}(\infty)=\operatorname{Spec}(\overline{A})\setminus\operatorname{Spec}(B)\cong\mathbb{A}^1$, scheme-theoretically. Therefore, $\widetilde{\pi}_{1}$ is an $\mathbb{A}^{1}$-bundle over $\mathbb{P}^{1}$ that satisfies the desired properties.
\end{proof}
It is known that every $\mathbb{A}^{1}$-bundle $\pi :V\to \mathbb{P}^{1}$, where $V$ is a affine nonsingular surface, can be realized as the complement $V=\mathbb{F}_{n}\setminus S$ of an ample section $S$, where $\mathbb{F}_{n}=\mathbb{P}\left(\mathcal{O}_{\mathbb{P}^{1}}\oplus \mathcal{O}_{\mathbb{P}^{1}}(-n)\right)$, $n\geq 0$, is a Hirzebruch surface, and the canonical projection $\tilde{\pi}_{n}:\mathbb{F}_{n}\to \mathbb{P}^{1}$ extends $\pi$. In this case, $S= C_{0}+\frac{1}{2}(S^{2}+n)F$, where $C_{0}$ is a section of $\tilde{\pi}_{n}$ with self-intersection $C_{0}^{2}=-n$, $F$ is a fiber of $\tilde{\pi}_{n}$, and $S^{2}\geq n+2$. Moreover, $n$ and $S^{2}$ are uniquely determined by $V$ and $\pi$, see \cite{Gizatullin-Danilov}, \cite[Lemma 5.5.1]{Miyanishi78}, \cite[Theorem 2.3]{Wright97}.
	
Danilov and Gizatullin proved that the isomorphism class of the complement $V=\mathbb{F}_{n}\setminus S$ of an ample section $S$ in a Hirzebruch surface $\mathbb{F}_{n}$ depends only on the self-intersection number $S^{2}$ and neither on $n$ nor on the choice of the section $S$, see \cite[Theorem 5.8.1]{Gizatullin-Danilov}, \cite[Corollary 4.8]{Cassou-Russell07}. These affine surfaces $V$ are called Danilov-Gizatullin surfaces of index $S^{2}$.
	
The Picard group of the Hirzebruch surface $\mathbb{F}_{n}$ is $\operatorname{Pic}(\mathbb{F}_{n})\cong \mathbb{Z}^{2}$, freely generated by the class of $C_{0}$ and the class of a fiber $F$ of $\tilde{\pi}_{n}$, and $C_{0}^{2}=-n$, $F^{2}=0$ and $C_{0}F=1$. The canonical divisor is given by $K_{\mathbb{F}_{n}}=-2C_{0}-(n+2)F=-2S+(S^{2}-2)F$. Therefore, the Picard group of the Danilov-Gizatullin surface $\mathbb{F}_{n}\setminus S$ is freely generated by the class of $F\vert_{\mathbb{F}_{n}\setminus S}$, $\operatorname{Pic}(\mathbb{F}_{n}\setminus S)\cong \mathbb{Z}$, and the canonical divisor is given by $K_{\mathbb{F}_{n}\setminus S}=(S^{2}-2)F\vert_{\mathbb{F}_{n}\setminus S}$.
	
\begin{proposition}\label{prop_V_is_DG}
Let $A=\mathbb{C}[p,q]\subset B=\mathbb{C}[x,y]$ be an \'{e}tale extension such that $\overline{A}\neq B$ and $\overline{A}$ is a primary $\overline{A}$-submodule of $B$, then $\operatorname{Spec}(\overline{A})$ is a Danilov-Gizatullin surface of index 3. 
\end{proposition}
\begin{proof}
	It follows from Theorem \ref{thm_p1_bundle} that $\operatorname{Spec}(\overline{A})$ is an $\mathbb{A}^{1}$-bundle over $\mathbb{P}^{1}$. By the observations above it is in particular a Danilov-Gizatullin surface $\operatorname{Spec}(\overline{A})\cong\mathbb{F}_{n}\setminus S$.  On the other hand, we know by Corollary \ref{cor_canonical_divisor} that the Picard group of $\operatorname{Spec}(\overline{A})$ is generated by the canonical divisor $K_{\operatorname{Spec}(\overline{A})}$. It follows that $(S^{2}-2)F\vert_{\mathbb{F}_{n}\setminus S}=F\vert_{\mathbb{F}_{n}\setminus S}$, this forces $S^{2}=3$.
\end{proof}
	
Now, the problem is whether a Danilov-Gizatullin surface (equivalently, an $\mathbb{A}^{1}$-bundle over $\mathbb{P}^{1}$) can factorizes an étale morphism of the complex affine plane. This is precisely the situation Wright conjectured should not happen. \emph{Geometric Formulation of Wright's conjecture}: let $V$ be an affine variety which is an $\mathbb{A}^{1}$-bundle over $\mathbb{P}^{1}$, $U=V\setminus F$, where $F$ is a fiber in $V$. There does not exist $f:V\to \mathbb{A}^{2}$ such that $f\vert_{U}$ is étale, \cite[Conjecture 3.2]{Wright97}.

To address his conjecture, Wright provides an explicit description of the coordinate ring of an affine $\mathbb{A}^{1}$-bundle over $\mathbb{P}^{1}$, that is, he describes $\Gamma(V)$ as a subring of $\mathbb{C}[x,y]$, corresponding to the containment of $U\cong \mathbb{A}^{2}$ in $V$. Let's recall his construction.

\begin{theorem}\cite[Theorem 3.1]{Wright97}\label{Wright_coord}
Let $V$ be an affine variety which is an $\mathbb{A}^{1}$-bundle over $\mathbb{P}^{1}$ with structure map $\pi: V\to \mathbb{P}^{1}$. Pick $x$ such that the function field of $\mathbb{P}^{1}$ is $\mathbb{C}(x)$, and let $U_{0}=\pi^{-1}\left(\operatorname{Spec}\mathbb{C}[x]\right)$, $U_{1}=\pi^{-1}\left(\operatorname{Spec}\mathbb{C}[x^{-1}]\right)$. Then $V=U_{0}\cup U_{1}$ and there exist $y\in\Gamma(V)$ such that $U_{0}=\operatorname{Spec}\mathbb{C}[x,y]$, $U_{1}=\operatorname{Spec}\mathbb{C}[x',y']$ where
\begin{align*}
x'=x^{-1},\quad y'=x^{m}y+\alpha_{1}x^{m-1}+\alpha_{2}x^{m-2}+\cdots+\alpha_{m-1}x,
\end{align*}
where $m\geq 2$ and $\alpha_{1}, \ldots, \alpha_{m-1}\in \mathbb{C}$ not all zero. Moreover
\begin{align*}
\Gamma(V)=\mathbb{C}[t_{0},t_{1},\ldots,t_{m}],
\end{align*}
where
\begin{equation}
\begin{aligned}\label{eq_coord}
t_{0}&=y,\\ 
t_{1}&=xy,\\ 
t_{2}&=x^{2}y+\alpha_{1}x,\\  t_{3}&=x^{3}y+\alpha_{1}x^{2}+\alpha_{2}x,\\ 
\ldots\\ 
t_{m}&=x^{m}y+\alpha_{1}x^{m-1}+\alpha_{2}x^{m-2}+\cdots+\alpha_{m-1}x\;(=y'). 
\end{aligned}
\end{equation}
Conversely, every ring of the form $\mathbb{C}[t_{0},t_{1},\ldots,t_{m}]$ above can be made into the coordinate ring of an $\mathbb{A}^{1}$-bundle over $\mathbb{P}^{1}$.
\end{theorem} 

In view of this description, Wright's conjecture can be rephrased as follows. \emph{Algebraic formulation of Wright's conjecture}: there does not exist a pair of polynomials $p,q\in \mathbb{C}[t_{0},t_{1},\ldots,t_{m}]\subset \mathbb{C}[x,y]$, where $t_{0}, t_{1}, \ldots, t_{m}$ are as in (\ref{eq_coord}) ($\alpha_{1}, \ldots, \alpha_{m-1}\in \mathbb{C}$ not all zero), with $\frac{\partial(p,q)}{\partial(x,y)}$ non-vanishing (i.e., constant) on $\mathbb{A}^{2}$, \cite[Conjecture 3.2]{Wright97}.

\begin{remark}\label{rmk_number_of_generators}
	The number of generators of the algebra $\Gamma(V)$, that is, the integer $m+1$, is completely determined by the $\mathbb{A}^{1}$-bundle $\pi: V\to \mathbb{P}^{1}$. In fact, let $V=\mathbb{F}_{n}\setminus S$ be the corresponding Danilov-Gizatullin surface, then $m=S^{2}$, see \cite[p. 599, above the Conjecture 3.2]{Wright97} and \cite[Proposition 3.5]{Cassou-Russell07}.
\end{remark}

Wright succeeded in proving the following case of his conjecture.

\begin{theorem}\cite[Theorem 3.3]{Wright97}\label{Wright_thrm_coef}
	Wright's conjecture is true in the case where the coefficient $\alpha_{1}$ is non-zero.
\end{theorem}

Putting together the previous results we obtain the following explicit description of the integral closure of a non-integral étale extension whose integral closure is primary.

\begin{proposition}\label{prop_description_integral_closure}
Let $A=\mathbb{C}[p,q]\subset B=\mathbb{C}[x,y]$ be an \'{e}tale extension such that $\overline{A}\neq B$ and $\overline{A}$ is a primary $\overline{A}$-submodule of $B$, then $\overline{A}=\mathbb{C}[y, xy, x^{2}y, x^{3}y+\alpha x]$, where $\alpha\in \mathbb{C}$, $\alpha\neq 0$. In particular $p,q\in \mathbb{C}[y, xy, x^{2}y, x^{3}y+\alpha x]$.
\end{proposition}
\begin{proof}
By Theorem \ref{thm_p1_bundle} we have that $V=\operatorname{Spec}(\overline{A})$ admits the structure of an  $\mathbb{A}^{1}$-bundle over $\mathbb{P}^{1}$ such that the open embedding of $\operatorname{Spec}(B)\cong \mathbb{A}^{2}$ in $\operatorname{Spec}(\overline{A})$ coincides with the complement of the fiber over the infinity and $\pi\vert_{\operatorname{Spec}(B)}:\mathbb{A}^{2}\to \mathbb{P}^{1}\setminus\{\infty\}=\mathbb{A}^{1}$ is one of the standard projections. Moreover, in Proposition \ref{prop_V_is_DG} we proved that $V$ is a Danilov-Gizatullin surface of index 3, therefore, $m=3$ by Remark \ref{rmk_number_of_generators}. It now follows from Wright's Theorem \ref{Wright_coord} that $p, q \in \Gamma(V)=\mathbb{C}[y, xy, x^{2}y+\alpha_{1}x, x^{3}y+\alpha_{1}x^{2}+\alpha_{2}x]$, where $\alpha_{1}, \alpha_{2}\in \mathbb{C}$ cannot be both zero. By Wright's Theorem \ref{Wright_thrm_coef} we have that $\alpha_{1}=0$.
\end{proof}

At this point we are left with the following case of Wright's conjecture: there does not exist a pair of polynomials $p,q\in \mathbb{C}[y, xy, x^{2}y, x^{3}y+\alpha x]$, $\alpha\in \mathbb{C}$, $\alpha\neq 0$ such that $\frac{\partial(p,q)}{\partial(x,y)}$ is constant. This is precisely the simplest unknown case of Wright's conjecture \cite[pag. 601]{Wright97}. We are going to use Proposition \ref{prop_description_integral_closure} and Wright's observation that $\mathbb{C}[y, xy, x^{2}y, x^{3}y+\alpha x]$ is a graded algebra with $\textrm{deg}\;x=-1$ and $\textrm{deg}\;y=2$ to prove that if it is always the case that $\overline{A}$ is a primary $\overline{A}$-submodule of $B$ then every two-dimensional étale extension is integral.

With respect to this grading the generators of $\mathbb{C}[y, xy, x^{2}y, x^{3}y+\alpha x]$ have degrees: $\textrm{deg}\;y=2$, $\textrm{deg}\;xy=1, \textrm{deg}\;x^2y=0$ and $\textrm{deg}\;(x^3y+\alpha x)=-1$. Let $f(x,y)\in \mathbb{C}[y, xy, x^{2}y, x^{3}y+\alpha x]$ be a homogeneous polynomial of weighted degree $-m<0$. We can write:
\begin{align*}
f(x,y)&=\sum_{2l_{1}+l_{2}-l_{4}=-m}f_{l_{1}l_{2}l_{3}l_{4}}y^{l_{1}}(xy)^{l_{2}}(x^2y)^{l_{3}}(x^3y+\alpha x)^{l_{4}}, \quad \quad l_{i}\in \mathbb{Z}_{\geq 0}, \quad f_{l_{1}l_{2}l_{3}l_{4}}\in \mathbb{C},\\
f(x,y)&=\sum_{2l_{1}+l_{2}-l_{4}=-m}f_{l_{1}l_{2}l_{3}l_{4}}y^{l_{1}}(xy)^{l_{2}}(x^2y)^{l_{3}}(x^3y+\alpha x)^{2l_{1}+l_{2}+m},\\
f(x,y)&=(x^3y+\alpha x)^{m}\sum_{2l_{1}+l_{2}-l_{4}=-m}f_{l_{1}l_{2}l_{3}l_{4}}(x^{6}y^{3}+2\alpha x^{4}y^{2}+\alpha^{2}x^{2}y)^{l_{1}}(x^{4}y^{2}+\alpha x^{2}y)^{l_{2}}(x^2y)^{l_{3}},
\end{align*}
then $f(x,y)=(x^3y+\alpha x)^{m}g(z)$, where $g(z)\in\mathbb{C}[z]$ and $z=x^{2}y$.

For the next definition, let's consider the usual grading, that is, $\textrm{deg}\;x=1$ and $\textrm{deg}\;y=1$. We say a polynomial $p(x,y)\in\mathbb{C}[x,y]$ of total degree $n$ is regular in $x$ if it contains a term in $x^{n}$. It is well known that for any polynomial $p(x,y)$ it is always possible to find an invertible linear transformation $x=av+bw$, $y=cv+dw$, ($a,b,c,d\in \mathbb{C}$, $ad-cd\neq 0$) such that $r(v,w)=p(av+bw,cv+dw)$ is regular in $v$ and $w$, see \cite[(2.7), p. 5]{Lefschetz_1953}.

\begin{lemma}\label{lmm-non_regular_element}
The algebra $\mathbb{C}[y, xy, x^{2}y, x^{3}y+\alpha x]$, $\alpha\in \mathbb{C}$, $\alpha\neq 0$ does not contain polynomials that are regular in both variables.
\end{lemma}

\begin{proof}
Suppose $p(x,y)\in\mathbb{C}[y, xy, x^{2}y, x^{3}y+\alpha x]$ is a polynomial of total (usual) degree $n$ that is regular in both variables. Then in particular $p(x,y)$ contains a term in $x^{n}$ and consequently $p(x,y)$ has a non-zero homogeneous component of weighted degree $-n<0$. On the other hand, we already observe that all homogeneous polynomials of weighted degree $-n<0$ in $\mathbb{C}[y, xy, x^{2}y, x^{3}y+\alpha x]$ are of the form $f(x,y)=(x^3y+\alpha x)^{n}g(z)$, where $g(z)\in\mathbb{C}[z]$ and $z=x^{2}y$. This is a contradiction because $p(x,y)$ is a polynomial of total (usual) degree $n$.  
\end{proof}
	
\begin{theorem}\label{thm-etale_implies_integral}
Suppose that $\overline{A}$ is a primary $\overline{A}$-submodule of $B$ for every étale extension $A=\mathbb{C}[p,q]\subset B=\mathbb{C}[x,y]$. Then every \'{e}tale extension of polynomial rings  $\mathbb{C}[p,q]\subset \mathbb{C}[x,y]$ is integral.
\end{theorem}
\begin{proof}
	Suppose that there exist an \'{e}tale extension $A=\mathbb{C}[p,q]\subset B=\mathbb{C}[x,y]$ that is not integral, that is, $\overline{A}\neq B$ (equivalently the corresponding étale morphism $(p,q):\mathbb{A}^{2}=\operatorname{Spec}(B)\to \mathbb{A}^{2}=\operatorname{Spec}(A)$ is not injective, \cite[Theorem 46]{Wang80}). Choose an invertible linear substitution $x=av+bw$, $y=cv+dw$, ($a,b,c,d\in \mathbb{C}$, $ad-cd\neq 0$) such that $r(v,w)=p(av+bw,cv+dw)$ is regular in both variables $v$ and $w$. After composing the morphism $(p,q)$ with this linear automorphism we obtain another \'{e}tale extension $\mathbb{C}[r,s]\subset \mathbb{C}[v,w]$ that is also non-integral (equivalently an étale morphism $(r,s):\mathbb{A}^{2}=\operatorname{Spec}(\mathbb{C}[v,w])\to \mathbb{A}^{2}=\operatorname{Spec}(\mathbb{C}[r,s])$ that is not injective), where $r(v,w)=p(av+bw,cv+dw)$ and $s(v,w)=q(av+bw,cv+dw)$. By hypothesis we have that $\overline{\mathbb{C}[r,s]}$ is a primary $\overline{\mathbb{C}[r,s]}$-submodule of $\mathbb{C}[v,w]$, applying Proposition \ref{prop_description_integral_closure} to the non-integral étale extension $\mathbb{C}[r,s]\subset \mathbb{C}[v,w]$ we have that $r,s\in \mathbb{C}[w, vw, v^{2}w, v^{3}w+\beta v]$, where $\beta\in \mathbb{C}$, $\beta\neq 0$. By construction the polynomial $r$ is regular in $v$ and $w$, however, by Lemma \ref{lmm-non_regular_element} the algebra $\mathbb{C}[w, vw, v^{2}w, v^{3}w+\beta v]$ does not contain polynomials regular in both variables, a contradiction. Hence $\overline{A}=B$.
\end{proof}

\bibliographystyle{abbrv}
\bibliography{bibliografia}

\end{document}